\documentclass{amsart}

\usepackage{amssymb,amsthm,amsmath}

\usepackage{hyperref}

\newtheorem{teo}{Theorem}
\newtheorem{lem}{Lemma}
\newtheorem{cor}[teo]{Corollary}
\theoremstyle{remark}
\newtheorem{rmk}{Remark}[section]

\newcommand{\Z}{\mathbb{Z}}

\begin{document}

\title{On Egyptian Fractions}


\author[M. Bello-Hern\'andez]{Manuel Bello Hern\'andez
\address{Dpto. de Matem\'aticas y Computaci\'on,
Universidad de La Rioja,
Edif. J. L. Vives, Calle Luis de Ulloa s/n,
26004 Logro\~no, Spain}
\email{mbello@unirioja.es}}
\thanks{This research was supported in
part from `Ministerio de Ciencia y Tecnolog\'{\i}a', Project
MTM2009-14668-C02-02}

\author[M. Benito]{Manuel Benito
\address{Instituto Sagasta, Glorieta del Doctor Zub{\'\i}a, s/n,
26003, Logro{\~n}o, La Rioja, Spain}
\email{mbenit8@palmera.pntic.mec.com}}

\author[E. Fern\'andez]{Emilio Fern\'andez
\address{Instituto Sagasta, Glorieta del Doctor Zub{\'\i}a, s/n,
26003, Logro{\~n}o, La Rioja, Spain}
\email{e.fernandez@iessagasta.com}}





\begin{abstract}We find a polynomial in three variables whose values at nonnegative integers satisfy the Erd\H{o}s-Straus Conjecture. Although the perfect squares are not covered by these values, it allows us to prove that there are arbitrarily long sequence of consecutive numbers satisfying the Erd\H{o}s-Straus Conjecture.  We conjecture that the values of this polynomial include all the prime numbers of the form $4q+5$, which is checked up to $10^{14}$. A greedy-type algorithm to find an Erd\H{o}s-Straus decomposition is also given; the convergence of this algorithm is proved for a wide class of numbers. Combining this algorithm with the mentioned polynomial we verify that all the natural numbers  $n$, $2\le n\le 2\times 10^{14}$, satisfy the Ed\H{o}s-Straus Conjecture.
\end{abstract}

\maketitle


\section{Introduction}

One of the most famous conjecture on Egyptian fractions is the  Erd\H{o}s-Straus Conjecture (ESC): {\it Given a positive integer $n\ge 2$ there exist positive integers $(x,y,z)\in\Z_{>0}^3$ such that
\begin{equation}\label{Erd-Str-Conj}
\frac4n=\frac1x+\frac1y+\frac1z.
\end{equation}
}In this case we say that $n$ is an Erd\H{o}s-Straus' number and refer to (\ref{Erd-Str-Conj}) as an Erd\H{o}s-Straus decomposition of $\frac4n$.  Sierpi\'nski and Schinzel extent to a more general conjeture replacing 4 in (\ref{Erd-Str-Conj}) by other fixed positive integer $m\ge 4$. The roots of these lie in the minimum number of Egyptian fractions needed to decompose a fraction as sum of Egyptian fractions (see, for example, \cite{Els-Tao}, \cite{Erd}, \cite{Mar}, \cite{Mor}, \cite{Sie} and \cite{Vau}). ESC has been verified for all integer up to a bound by many authors: Straus, Bernstein \cite{Ber}, Yamamoto \cite{Yam}, Swett \cite{Swe}, etc. Swett has checked ESC for all $n\le10^{14}$. If $n$ is an Erd\"{o}s-Straus' number, then it also holds true for integers which are divisible by $n$; moreover, $\frac4{4q+3}=\frac1{q+1}+\frac1{(q+1)(4q+3)}$ and
if we have the factorization $n=a\cdot b\cdot c$,
$\frac1n=\frac1{a(a+b)c}+\frac1{b(a+b)c}.$
Therefore, ESC should be only checked for primes $n=4q+1$.

Because
\begin{equation}\label{eqTipoI}
\frac4{abc-a-b}=\frac{1}{a\frac{bc-1}4}+\frac{1}{a(ac-1)\frac{bc-1}4}+\frac{1}{(ac-1)\frac{bc-1}4n},
\end{equation}
when $bc\equiv 1\, \textnormal{(mod 4)}$, the numbers $n=abc-a-b$ are Erd\H{o}s-Straus' numbers. Moreover, if we replace the condition $bc\equiv 1\, \textnormal{(mod 4)}$ by $bc\equiv 1$ (mod $m$), then the Sierpi\'nski and Schinzel conjecture holds for $n=abc-a-b$. Since we are interested in $n=4q+1$, let $n=p(\alpha,\beta,\gamma)$ be  the polynomial $p:\Z_{\ge0}^3\to\Z$  given by
\begin{equation}\label{pol}
p(\alpha,\beta,\gamma)=(\alpha+1)(4\beta+3)(4\gamma+3)-(\alpha+1)-(4\beta+3).
\end{equation}

This parametric solution of ESC let us to prove the following result:
\begin{teo}\label{teoConsESN}There are arbitrarily long sequence of consecutive numbers satisfying Erd\H{o}s-Straus Conjecture.
\end{teo}
We have checked that the set
\[
\mathcal{N}_1=\left\{n\in\Z_{>0}:\exists (\alpha,\beta,\gamma)\in \mathbb{Z}_{\ge0}^3,\, n=p(\alpha,\beta,\gamma)\right\}
\]
contains all the prime numbers of the form $n=4q+5,\, q\in\Z_{\ge0},$ $n<10^{14}$.  We prove, see  Lemma \ref{lemNoSquares} below, that $\mathcal{N}_1$ does not contain the  perfect squares.

In Section \ref{sectionAuxiliaryRelations} we list several relations for ESC. One interesting conclusion in this section is given in Lemma \ref{lemNoResClas}: it is impossible to generate ``naturally'' a finite  number of congruence classes satisfying ESC which contain all the prime.  Section \ref{sectionConsecutiveNumbers} contains a proof of Theorem \ref{teoConsESN}. In Section \ref{SectionAditional} we prove that the set $\mathcal{N}_1$ does not contain the perfect squares and several other additional remarks on ESC. We describe a fast algorithm to construct a decomposition of $\frac4n$ as sum of three Egyptian fractions and we prove the convergence of that algorithm for a large class of numbers in the last Section \ref{SectionAlgorithm}.


\section{Auxiliary relations}\label{sectionAuxiliaryRelations}


The Erd\H{o}s-Straus Conjecture shows an arithmetic property of the natural numbers (see Lemma \ref{LemConjQ} below). We can split any fraction $\frac2n$ as sum of two Egyptian fraction but the same does not hold for $\frac3n$ for all $n\in\Z_{>0}$. It is very well known that the equation
$\frac{3}{n}=\frac1x+\frac1y$ has solutions $(x,y)\in\Z_{>0}^2$ if and only if $n$ has a divisor $m$ with either $m\equiv0$ or $m\equiv 2$ (mod 3); it is equivalent to $n\not\equiv 1$ (mod 6).

The following Lemma is very well known; see, for example \cite{Mor}, p. 287, or \cite{Yam}. We give a proof for an easy reading.


\begin{lem}\label{BasicCharac} Let $n$ be a prime; $n$ is an Erd\H{o}s-Straus' number
if and only if there exists $(a,b,c,d)\in\Z_{>0}^4$ such that some of the following conditions holds:
\begin{align}\label{condNecSufErdosStraus1}
(4a b c-1)d&=(a+b)n,\\ \label{condNecSufErdosStraus2}\quad(4abc-1)d&=a n+b.
\end{align}
\end{lem}

\begin{proof} If $(4abc-1)d=(a+b)n$ or $(4abc-1)d=a n+b$ holds, dividing these equations by $abcd n$, we have respectively
\begin{align}\label{descomposicionSze1}
\frac4n&=\frac1{abc n}+\frac1{bcd}+\frac1{acd},\\ \label{descomposicionSze2}
\frac4n&=\frac1{abc n}+\frac1{bcd}+\frac1{acd n}.
\end{align}

On the other hand, if (\ref{Erd-Str-Conj})
holds, then $4 xyz=n(xy+yz+zx).$ Since $n$ is prime, $n$ divides $x,$ $y$ or $z$. Of course, $n$ does not divide  all the three numbers $x,y,z$ because this trivially yields the contradiction $4=\frac1{x/n}+\frac1{y/n}+\frac1{z/n}$ with $x/n,y/n, z/n$ positive integers.  Hence, we can assume without lost of generality $x=an$. Thus (\ref{Erd-Str-Conj}) is equivalent to
\begin{equation}\label{ecuacAux1}
\frac{4a-1}{n a}=\frac1y+\frac1z \Leftrightarrow\frac{1}{n a}=\frac1{(4a-1)y}+\frac1{(4a-1)z}.
\end{equation}
Since $n$ is prime and $(4a-1,a)=1,$ we  have two cases: $(4a-1,n)=1$, and $(4a-1,n)=n$.

In the first case,  there exist $a_1,a_2,a_3\in\Z_+$ such that $a=a_1a_2a_3$, $(na_1,a_2)=1$, and
\begin{equation}\label{ecuacAux2}
(4a-1)y=n a_1(na_1+a_2)a_3,\quad (4a-1)z=a_2(na_1+a_2)a_3.
\end{equation}
Because $(na,4a-1)=(n a_1 a_2 a_3,4a-1)=1$, there exist $\alpha,\beta$ such that
$y=\alpha n a_1a_3,\, z=\beta a_2a_3.$
From (\ref{ecuacAux2}), we have
$(4a-1)\alpha=(na_1+a_2)=(4a-1)\beta,$
it means $\alpha=\beta$. Therefore, $A=a_1,B=a_2,C=a_3$, and $D=\alpha$ satisfy
\[
(4 ABC-1)D=(n A+B).
\]

Now we consider the second case $(4a-1,n)=n$. So, there exists $j$ such that
\begin{equation}\label{ecuacAux3}
4a-1=j n.
\end{equation}
Setting this expression in (\ref{ecuacAux1}), we have $\frac1{a}=\frac1{j y}+\frac1{j z}.$
Thus, there exist $a_1,a_2,a_3\in\Z_+$ such that $a=a_1a_2a_3$, $(a_1,a_2)=1$ and
\begin{equation}\label{ecuacAux4}
j y=a_1(a_1+a_2)a_3,\quad j z=a_2(a_1+a_2)a_3.
\end{equation}
Since $(j,a)=(j,a_1a_2a_3)=1$, there exist $\alpha,\beta$ such that $y=\alpha a_1a_3$, $z=\beta a_2a_3$. Setting these expression in  (\ref{ecuacAux4}), we have $j\alpha=a_1+a_2=j\beta,$ it is $\alpha=\beta$. Multiplying in (\ref{ecuacAux3}) by $\alpha$, we obtain $A=a_1,\,B=a_2,\,C=a_3,\, D=\alpha$ such that
\[
(4 ABC-1)D=(A+B)n.
\]
\end{proof}

The relations (\ref{condNecSufErdosStraus1}) and (\ref{condNecSufErdosStraus2}) are equivalent to (\ref{descomposicionSze1}) and (\ref{descomposicionSze2}), respectively. Since in the first decomposition of $\frac4n$, $n$ divides one denominator but is coprime to the others, while in the second it is only coprime to one of them, these are respectively refereed as Type I and Type II decomposition of $\frac4n$ (see \cite{Els-Tao}). Observe that it is not possible for all three of $x,y,z$ in (\ref{Erd-Str-Conj}) to be divisible by $n$.

Changing variables in (\ref{condNecSufErdosStraus1}) and (\ref{condNecSufErdosStraus2}) the following lemma follows.


\begin{lem}\label{parametrization}
\begin{enumerate}
\item Let $n$ be a prime. There exist positive integers $a,b,c,d$ such that (\ref{condNecSufErdosStraus1}) holds if and only if there are positive integers $\alpha, \beta,\gamma,\delta$ such that
\begin{equation}\label{eqTipoTres}
\delta n=(4\alpha\beta\gamma\delta-1)-4 \alpha^2\gamma.
\end{equation}
\item  Let $n\in\Z$, $n\ge 2$. There exist positive integers $a,b,c,d$ such that (\ref{condNecSufErdosStraus2}) holds if and only if there are positive integers $\alpha, \beta,\gamma,\delta$ satisfying
\begin{equation}\label{eqTipoDos}
n=(4\alpha\beta\gamma-1)\delta-4 \beta^2\gamma.
\end{equation}
\end{enumerate}
\end{lem}

\begin{proof}
\begin{enumerate}
\item If $n$ is a prime and (\ref{condNecSufErdosStraus1}) holds, then $d$ divides $a+b$. Set $e=\frac{a+b}{d}$; hence, $b=de-a$ and (\ref{condNecSufErdosStraus1}) yields $en=(4acde-1)-4a^2c.$ Setting $\delta=e$, $\alpha=a$, $\beta=d$, and $\gamma=c$, we obtain (\ref{eqTipoTres}).
\item The relation (\ref{condNecSufErdosStraus2}) is equivalent to $b+d$ is divisible by $a$ and $n+s=4bcd$, where $s=\frac{b+d}{a}\Leftrightarrow d=as-b$. Setting $\alpha=a,\beta=b,\gamma=c,\delta=s$ we obtain immediately  (\ref{eqTipoDos}).
\end{enumerate}
\end{proof}

Equation (\ref{eqTipoTres}) implies that $\delta n$ is an Erd\H{o}s-Straus' number with Type I and Type II decomposition, but it is trivial because if $n\equiv 1$ modulo $4$ it follows $\delta\equiv-1$ modulo $4$. The parametric relation  (\ref{eqTipoDos}) appears implicitly in \cite{Vau}. Setting $\beta=\gamma=1$ in (\ref{eqTipoDos}), it follows  that if $n+4$ has a divisor congruent to $3$ modulo $4$, then $n$ is an Erd\H{o}s-Straus' number. So by the Landau prime ideal theorem (see \cite{Lan} or \cite{Mon-Vau} pp. 266–-268) the set of positive integers which are no Erd\H{o}s-Straus' numbers has zero density. A sharp estimate  on the density of Erd\H{o}s-Straus' numbers can be found in \cite{Vau} (see also \cite{Els-Tao}, \cite{Hua-Vau1}, \cite{Hua-Vau2}, and \cite{Jia}).

In \cite{Yam} Yamamoto (see also \cite{Sch}) proves that a perfect square $n$ does not satisfy  neither (\ref{condNecSufErdosStraus1}) nor (\ref{condNecSufErdosStraus2}) with some restriction in the parameters (see the paragraphs at the begging of Subsection  \ref{SectionNoCuadrados} below), so fixing the parameters $a,b,c,d$ in these equations we can not generated a complete residue system. The following lemma emphasis in this remark without use that result of Yamamoto.

\begin{lem}\label{lemNoResClas} We can not generate a finite number of congruence classes containing all the primes of the form $4q+5$ fixing in (\ref{condNecSufErdosStraus1}) or (\ref{condNecSufErdosStraus2}) three of the four parameters in a finite subset of $\Z_{> 0}^3$ and the remain parameter free in $\Z_{>0}$.
\end{lem}

\begin{proof} Let $S_a=\{(b,c,d)\in\Z_{> 0}^3\}$ denote a finite subset of $\Z_{> 0}^3$ of values taken by $(b,c,d)$. From (\ref{condNecSufErdosStraus1}) or (\ref{condNecSufErdosStraus2}), when $a$ is free taking any positive number, each fixed vector $(b,c,d)\in S_a$ may generate a residue class of Erd\H{o}s-Straus' numbers  $n$. Equivalently  we define $S_b,S_c$ and $S_d$.  We fixe four such subset $S_a,S_b,S_c,S_d$ and first prove that there exist infinite prime numbers $n$ which can not be generated by (\ref{condNecSufErdosStraus2}) with parameters $(a,b,c,d)$ with three of them in the corresponding set $S_a, S_b, S_c,$ or $S_d$.

The numbers $n$ generate by (\ref{condNecSufErdosStraus2}) with $(b,c,d)\in S_a$ are
a finite set because  $a|(b+d)$. Moreover, observe that in the proof of (\ref{eqTipoDos}) in Lemma \ref{parametrization}, the numbers $n$ generate by (\ref{condNecSufErdosStraus2}) or equivalently by (\ref{eqTipoDos})  with $(b,c,d)\in S_c$ are given by
\[
n=4(\alpha\delta-\beta)\beta\gamma-\delta=4bcd-\frac{b+d}{a},\quad (a,b,d)\in S_c,
\]
where $\alpha=a$, $\beta=b$, $\gamma=c$, and $\delta=\frac{b+d}{a}$. Taking $T_c=\textnormal{lcm}\{bd:(a,b,d)\in S_c\}$, the numbers  in $\{4T_ct+1:t\in\Z_{>0}\}$ (in particular the prime numbers in $\{4T_ct+1:t\in\Z_{>0}\}$) can not be generated by (\ref{eqTipoDos}) with $(a,b,d)\in S_c$, here $\textnormal{lcm}\{bd:(a,b,d)\in S_c\}$ denotes the least common multiple of the numbers $bd$ when $(a,b,d)$ runs through the set $S_c$. Suppose for contradiction that for a given $t\in\Z_{>0}$ there exist $(a_0,b_0,d_0)\in S_c$ and $c\in\Z_{>0}$ such that
\[
4T_ct+1=4b_0cd_0-\frac{b_0+d_0}{a_0}\Leftrightarrow 1=4b_0d_0(c-e)-\frac{b_0+d_0}{a_0},
\]
where $e\in\Z_{>0}$, but it is a contradiction since $1$ is not an Erd\H{o}s-Straus' number.

Using the same arguments and notations as before, when $(a,b,c)\in S_d$, then according to the proof of (\ref{eqTipoDos}) in Lemma \ref{parametrization}, $(\alpha,\beta,\gamma)$ lies in a finite subset $S_{\delta}$ of $\Z_{> 0}^3$,  and $\delta\in\Z_{>0}$. Taking $T_d=\textnormal{lcm}\{(4\alpha\beta\gamma-1):(\alpha,\beta,\gamma)\in S_{\delta}\}$, the numbers in $\{4T_d t+1:t\in\Z_{>0}\}$ can not be generated by (\ref{eqTipoDos})  with $(a,b,c)\in S_d$. Suppose for a $t\in\Z_{>0}$ there exists $(\alpha_0,\beta_0,\gamma_0)\in S_\delta$ and $\delta\in \Z_{>0}$ such that
\[
4T_dt+1=(4\alpha_0\beta_0\gamma_0-1)\delta-4\beta_0^2\gamma_0\Leftrightarrow 1=(4\alpha_0\beta_0\gamma_0-1)(\delta-4e)-4\beta_0^2\gamma_0
\]
for some $e$, again it is not possible because $1$ is not an Erd\H{o}s-Straus' number. From the symmetry of (\ref{eqTipoDos}) in $b$ and $d$ the same conclusion holds for $S_b$ for numbers $\{4T_{b} t+1:t\in\Z_{>0}\}$, where $T_{b}$ is defined as $T_d$.  Of course all the numbers
\[
\{4T_bT_cT_dt+1:t\in\Z_{>0}\}
\]
can not be generated by parameters $(a,b,c,d)$ with three of them in the corresponding set $S_a, S_b, S_c$ or $S_d$.

The same arguments work for (\ref{condNecSufErdosStraus1}). For example, one of the parameters $a$ or $b$ in (\ref{condNecSufErdosStraus1}) can not take values in an infinite subset of positive integers while other three belong to a finite subset. In fact, observe that (\ref{condNecSufErdosStraus1}) is symmetric in $a$ and $b$, and according to the proof of (\ref{eqTipoTres}) in Lemma \ref{parametrization}, the number $e=\frac{a+b}{d}$ divides $1+4a^2c$, so that if $(a,c,d)$ runs over a finite set, the numbers of divisors of  $1+4a^2c$ is finite. Moreover, the numbers $n$ generate by (\ref{condNecSufErdosStraus1}) or equivalently by (\ref{eqTipoTres})  with $(b,c,d)\in S_c$ do not contain $\{4T_c't+1:t\in\Z_{\ge 0}\}$ where $T_c'=\textnormal{lcm}\{abd:(a,b,d)\in S_c\}$. Suppose for contradiction that for a given $t$ there exist $(a_0,b_0,d_0)\in S_c$ and $c\in\Z_{>0}$ such that
\[
(a_0+b_0) (4T_c't+1)=(4a_0b_0c-1)d_0\Leftrightarrow (a_0+b_0) 1=(4a_0b_0(c-e)-1)d_0
\]
where $e=\frac{(a_0+b_0)T_c't}{a_0b_0d_0}$, but it is a contradiction since $1$ is not an Erd\H{o}s-Straus' number.
For $(a,b,c)\in S_d$, we can not generate the numbers $4T_d't+1$, where $T'_d=\textnormal{lcm}\{(4abc-1):(a,b,c)\in S_d\}$.

Therefore, equations (\ref{condNecSufErdosStraus1}) and (\ref{condNecSufErdosStraus2}) can not generate all the numbers
\begin{equation}\label{NoCovers}
\{4T_bT_cT_dT'_cT'_dt+1:t\in\Z_{>0}\}
\end{equation}
except possibly a finite number of them with three parameters of $a,b,c,d$ in the corresponding set $S_a, S_b, S_c$ or $S_d$.

\end{proof}

\begin{rmk} Note that using Yamamoto result we can actually prove that equations (\ref{condNecSufErdosStraus1}) and (\ref{condNecSufErdosStraus2}) can not generate all the numbers $\{4T_bT_cT_dT'_cT'_dt+n_0:t\in\Z_{>0}\}$ except possibly a finite number of them with three parameters of $a,b,c,d$ in the corresponding finite set $S_a, S_b, S_c$ or $S_d$, where $n_0$ is a quadratic residue modulo $4T_bT_cT_dT'_cT'_d$.

\end{rmk}

Many papers devote special attention to Type II decomposition since  parametric solutions of ESC are easy obtained for this case, see (\ref{eqTipoDos}). Our following lemma allows us to find a parametric solution (\ref{pol}) of Type I decomposition.

\begin{lem}\label{LemAuxUnDenMul} Let $n\in\Z_{>0}$. There exist positive integers $a,b,c,d$ such that (\ref{condNecSufErdosStraus1}) holds if and only if there are $x,t,\lambda$ such that
\begin{equation}\label{sisEcuaCar}
\left\{\begin{array}{l}
\frac{x\,n+t}{\lambda}\in\Z_{>0},\\ \frac{n+\lambda}{4\, x\, t}\in\Z_{>0}\end{array}\right.
\end{equation}
are satisfied.
\end{lem}

\begin{proof} Assume that there exist positive integers $x,t,\lambda,y,\, z$ such that
\begin{align*}
\frac{x\,n+t}{\lambda}=y,\, \frac{n+\lambda}{4\, x\, t}=z&\Leftrightarrow x\,n+t=y\,\lambda,\,\lambda=4\,z\, x\, t -n\\
&\Leftrightarrow (x+y)n=(4\, x\, y\, z-1)t.
\end{align*}
Therefore, taking $a=x,b=y,c=z,d=t$ we have $(4abc-1)d=(a +b)n$.
\end{proof}

\begin{rmk} If $n$ is an Erd\H{o}s-Straus' number and  (\ref{sisEcuaCar}) is satisfied for positive integers $x,t,\lambda$, then for all $j\in\Z_{\ge0}$,
$N=n+4xt\lambda\, j$ is also an Erd\H{o}s-Straus' number. In fact,
\[
\left\{\begin{array}{l}\frac{x\,N+t}{\lambda}= \frac{x\,(n+4xt\lambda\, j)+t}{\lambda}=\frac{x n+t}{\lambda}+4x^2tj\in\Z_+\\ \frac{N+\lambda}{4\, x\, t}=\frac{(n+4xt\lambda\, j)+\lambda}{4\, x\, t}=\frac{n+\lambda}{4\, x\, t}+\lambda\, j\in\Z_+\end{array}\right.
\]
In particular, the set of Erd\H{o}s-Straus' numbers is an open set in Furstenberg's topology (see \cite{Niv-Zuc-Mon}, p. 34).
\end{rmk}

\begin{rmk} The values
\[
n=p(\alpha,\beta,\gamma)=(\alpha+1)(4\beta+3)(4\gamma+3)-(\alpha+1)-(4\beta+3)
\]
satisfy (\ref{sisEcuaCar}) with $x=1$, $t=\alpha+1$, and $\lambda=4\beta+3$.
\end{rmk}

The parametric relation (\ref{pol}) for $p=4q+5$ is useful to rewrite for $q$. Since
\[
p(\alpha,\beta,\gamma)=4q(\alpha,\beta,\gamma)+5=4\left(((4\beta+3)\gamma+(3\beta+2))(\alpha+1)-(\beta+2)\right)+5,
\]
the following lemma follows:
\begin{lem}\label{Lem_pol_q} If there exists $(\alpha,\beta,\gamma)\in\Z_{\ge 0}$ such that
\begin{equation}\label{pol_q}
q=q(\alpha,\beta,\gamma)=((4\beta+3)\gamma+(3\beta+2))(\alpha+1)-(\beta+2),
\end{equation}
then the Erd\H{o}s-Straus Conjecture holds for $p=4q+5.$
\end{lem}

\section{Consecutive numbers satisfying ESC}\label{sectionConsecutiveNumbers}
Next, using the Chinese Remainder Theorem and  Lemma \ref{Lem_pol_q} we prove Theorem \ref{teoConsESN}. Actually, we prove that there exist arbitrarily long sequence of consecutive residues class such that $\frac4n$ has Type I decomposition for all $n$ in this residues class.

\textit{Proof of Theorem \ref{teoConsESN}.} Of course, to get Theorem \ref{teoConsESN} it is enough to consider ``consecutive'' numbers $n\equiv 1$ mod 4. Let $N$ be an arbitrary positive integer and let $A$ be a subset of $\Z_{>0}$ containing $N$ consecutive nonnegative integers, for example $A=\{0,1,\ldots, N-1\}$.  If  $\{\beta_1,\beta_2\}\subset A$, then the greatest common divisor of $4\beta_1+3$ and $4\beta_2+3$, $(4\beta_1+3,4\beta_2+3)$, is an odd number and
\begin{multline*}
(4\beta_1+3,4\beta_2+3)|(4\beta_1+3-(4\beta_2+3))=4(\beta_1-\beta_2)\\ \Rightarrow(4\beta_1+3,4\beta_2+3)|3(\beta_1-\beta_2) =3\beta_1+2-(3\beta_2+2).
\end{multline*}
Thus, by the Chinese Remainder Theorem, there is a natural number $T$ such that
\[
T\equiv 3\beta_j+2\quad (\text{mod}\, 4\beta_j+3),\quad \forall \beta_j\in A,
\]
i.e., there exist positive integers $\gamma_j$ such that
\[
T=(4\beta_j+3)\gamma_j+3\beta_j+2,\quad \forall \beta_j\in A.
\]
According to Lemma \ref{Lem_pol_q} all $n=4q+5$ with $q$ in the consecutive residue class $-(\beta_j+2)\,\, (\text{mod}\,\,T)$, $\beta_j\in A$, satisfy the Erd\H{o}s-Straus conjecture.

\hfill $\Box$

\begin{rmk} Moreover, we can also prove that there exists arbitrarily long sequence of consecutive numbers $n$ such that $\frac4n$ has Type II decomposition. Let $A$ be a set of natural numbers containing consecutive numbers, for each $a\in A$, we choose natural numbers $\beta(a)$ and $\gamma(a)$ such that $a=\beta(a)^2\gamma(a)$  (such representation is unique taking $\gamma(a)$ free of squares). We chose $T$ as the least common multiple of $(4\beta(a)\gamma(a)-1)$, where $a$ takes the values in $A$, i.e.
\[
T=\textnormal{lcm}\{(4\beta(a)\gamma(a)-1):a\in A\},\quad \textnormal{and}\quad \delta=\left\{\begin{array}{lr}1,&\textnormal{if }T\equiv 1 \textnormal{ mod }4,\\ 3 ,&\textnormal{if }T\equiv -1 \textnormal{ mod }4,\end{array}\right.
\]
then, using (\ref{eqTipoDos}) in Lemma \ref{parametrization},  it follows that all the fractions $\frac4{T\,\delta-4a},\, a\in A,$ have Type II decomposition.
\end{rmk}


\section{Additional remarks} \label{SectionAditional}

\subsection{ $q-$Conjecture}

\begin{lem}\label{LemConjQ} ESC is true if and only if for each $q\in\Z_{>0}$ there exists $(x,y,z)\in\Z_{\ge0}^3$ such that one of the following relations holds
\begin{eqnarray}
\label{ManoloEq1} q=1+3x+3y+4xy,\\ \label{ManoloEq2} q=5+5x+5y+4xy,\\ \label{ManoloEq3} q=q(x,y,z).
\end{eqnarray}
where $q(x,y,z)$ is defined by (\ref{pol_q}).
\end{lem}

\begin{proof} Because
\[
q=1+3x+3y+4xy\Leftrightarrow 4q+5=(4x+3)(4y+3),
\]
\[
q=5+5x+5y+4xy\Leftrightarrow 4q+5=(4(x+1)+1)(4(y+1)+1),
\]
all $q$ for composite numbers $4q+5$ satisfy (\ref{ManoloEq1}) or (\ref{ManoloEq2}).  We also have for $q=q(x,y,z)$ define by (\ref{pol_q}),
\[
q=q(x,y,z)=\frac{p(x,y,z) - 5}{4}.
\]
Hence if all $q\in\Z_{>0}$ satisfies (\ref{ManoloEq1}), (\ref{ManoloEq2}) or (\ref{ManoloEq3}), the value of $q$ for all prime numbers of the form $4q+5$ lies in $\mathcal{N}_1$ and the proof is concluded using Lemma \ref{Lem_pol_q}.
\end{proof}

\begin{rmk}
We have verified that $\mathcal{N}_1$ contains all primes of the form $n=4q+5< 10^{14}$. To this aim we have generated the classes of equivalence of numbers contained in $\mathcal{N}_1$ translating $t$ units each variable; i.e.
\begin{eqnarray*}
N_x=p(x,y,z)+f_1(y,z)t\overset{\textnormal{def}}=p(x,y,z)+4\frac{(4y+3)(4z+3)-1}4t,\\
N_y=p(x,y,z)+f_2(x,z)t\overset{\textnormal{def}}=p(x,y,z)+4((x+1)(4z+3)-1)t,\\
N_z=p(x,y,z)+f_3(x,y)t\overset{\textnormal{def}}=p(x,y,z)+4(x+1)(4y+3)t.
\end{eqnarray*}
Starting with $x,y,z\in\{0,1\}$ we obtain the following equivalence class
\[
5+8t,\,5+12t,\,13+20t,\,17+20t,\,13+28t,
37+52t,\quad t\in\Z_{\ge 0},
\]
then we sieve prime numbers congruent to 1 modulo 4 in these congruence classes. We are checked by computer calculations that the remaining primes $<10^{14}$ belong to $\mathcal{N}_1$.

\end{rmk}

{\bf Conjecture.} We conjecture that all the natural number $q\in\Z_{>0}$ can be written as one of the three above relations (\ref{ManoloEq1}), (\ref{ManoloEq2}) or (\ref{ManoloEq3}). We refer it as the $q-$Conjecture. According to the lemma above it is equivalent to all prime numbers of the form $4q+5$ lie in $\mathcal{N}_1.$

\subsection{$\mathcal{N}_1$ does not contain perfect squares}\label{SectionNoCuadrados}

Here we prove that $\mathcal{N}_1$ does not include perfect squares, we need to use Jacobi's symbol.  Let $p$ be an odd prime and $n\in \Z$ with $(n,p)=1$, Legendre's symbol is defined by $$
\left(\frac{n}{p}\right)=\left\{\begin{array}{rl}1,&\textnormal{if $n$ is a quadratic residue mod $p$,}\\ -1,&\textnormal{if $n$ is a quadratic non-residue mod $p$. }\end{array}\right.$$
Let the standard factorization of $m$ be $p_1p_2\cdots p_k$ where the $p_r$ may be
repeated. If $(n, m) = 1$ then Jacobi's symbol is defined in term of the Legendre symbol by
$\left(\frac{n}{m}\right)=\prod_{j=1}^k\left(\frac{n}{p_j}\right).$
The following properties of the Jacobi symbol are very well known: Let $m$ and $m'$ be positive odd integers.
(i.) If $n \equiv n'$ (mod $m$) and $(n, m) = 1$, then $\left(\frac{n}{m}\right) = \left(\frac{n'}{m}\right)$. (ii.) If $(n, m)= (n, m') = 1$, then $\left(\frac{n}{m}\right)\left(\frac{n}{m'}\right)=\left(\frac{n}{mm'}\right)$.
(iii.) If( $n,m) = (n',m) = 1$, then $\left(\frac{n}{m}\right)\left(\frac{n'}{m}\right)= \left(\frac{nn'}{m}\right)$. (iv.) If $n\equiv 3$ (mod $4$), $(-1/n)=-1$.
(v.) {\it Law of reciprocity.} Let $m$ and $n$ be odd coprime. Then
$\left(\frac{n}{m}\right)\left(\frac{m}{n}\right)=(-1)^{\frac{n-1}{2}\frac{m-1}{2}}.$
(vi.) If $n,d$ are coprime positive integers and  $m$ satisfies  $n\equiv-m $ (mod $d$), then $
\left(\frac{d}{n}\right)=\left(\frac{d}{m}\right),$ see \cite{Hua}, p. 305.
The Kronecker symbol is defined in term of Jacobi's symbol. It satisfies $\left(\frac{n}{m}\right)=0$ for $(n,m)\ge 2$.

In \cite{Yam}, Yamamoto observes that numbers $n$, satisfying $(4abc-1)d=(a +b)n$ for some positive integers $a,b,c,d$ with $(n,abd)=1$, are no perfect square. Observe the symmetry of $(4abc-1)d=(a +b)n$ in $a$ and $b$ and Yamamoto proves, using our notations, $\left(\frac{n}{4ad}\right)=-1$ for the Kronecker symbol. Moreover, if $(4abc-1)d=(a +b)n$, then $(4abc-1)d'=(a +b)n^2$ with $d'=dn$.  Our class $\mathcal{N}_1$ contains numbers $n$ which are not included in Yamamoto class. For example, it is the case of $n=2009,$ where $a=1$, $b=293$, $c=12$, and $d=42$, meanwhile $\alpha+1=42,$ $4\beta+3=7,$ and $4\gamma+3=7$. In fact, we have
\begin{eqnarray*}
2009=42\cdot7\cdot7-42-7,\\
2009(1+293)=(4\cdot1\cdot 293-1)42,\\
(2009,293\cdot42)=7.
\end{eqnarray*}

\begin{lem}\label{lemNoSquares} $\mathcal{N}_1$ does not contain perfect squares.\end{lem}

\begin{proof}
Let $n\in\mathcal{N}_1$. Then there exist nonnegative integers $\alpha,\beta,\gamma$ such that
\[
n+(4\beta+3)=(\alpha+1)\left((4\beta+3)(4\gamma+3)-1\right)\overset{{\it def}}{=}(\alpha+1)\tau.
\]
where $\tau=(4\beta+3)(4\gamma+3)-1$. Because of  $(4\beta+3,\tau)=1$, we also have  $(n,\tau)=1$ and $(n+\tau,\tau)=1$.
Using reciprocity law for Jacobi's symbol, we obtain
\begin{equation}\label{eqRecQuad}
\left(\frac{n}{n+\tau}\right)\left(\frac{n+\tau}{n}\right)=(-1)^{\frac{n-1}{2}\frac{n+\tau-1}{2}}=1,
\end{equation}
here we have used $n\equiv 1$ (mod 4). Since $n+\tau\equiv \tau$ (mod $n$), we have
$\left(\frac{n+\tau}{n}\right)=\left(\frac{\tau}{n}\right).$
Taking into account $n\equiv -(4\beta+3)$ (mod $\tau$) and property (vi) for the Jacobi's symbol cited above, we have
$\left(\frac{\tau}{n}\right)=\left(\frac{\tau}{4\beta+3}\right).$
Since $\tau\equiv -1$ (mod $4\beta+3$) and using Property (iv), we obtain
$\left(\frac{\tau}{4\beta+3}\right)= \left(\frac{-1}{4\beta+3}\right)=-1.$
Therefore, from (\ref{eqRecQuad}) we have
$\left(\frac{n}{n+\tau}\right)=-1,$
it yields $n$ is not a perfect square.
\end{proof}

Since $4(n^2+n-1)+5=(2n+1)^2$, and $4q(\alpha,\beta,\gamma)+5=p(\alpha,\beta,\gamma)$, where $p(\alpha,\beta,\gamma)$ and $q(\alpha,\beta,\gamma)$ are given by (\ref{pol}) and (\ref{pol_q}) respectively, the Lemma \ref{lemNoSquares} gives immediately the following result:
\begin{cor} The numbers $n^2+n+\beta+1,\, (n,\beta)\in \Z_{\ge 0}^2,$ have no divisors
congruent to $3\beta+2$ modulo $4\beta+3$.
\end{cor}

\subsection{ $q-$Strong Conjecture}
Lemma \ref{lemNoSquares} and computer calculations let us to formulate the following conjecture which implies ESC and has been checked up to $2\times 10^{12}$.

{\bf $q-$Strong Conjecture:} The positive integers $\Z_{>0}$ is disjoint union of the three sets:
\begin{itemize}
\item $A=\{q:\exists n\in\Z_{>0}, \, q=n^2+n-1\}$.
\item $B=\{q:\exists(\alpha,\beta,\gamma)\in\Z_{\ge0}^3,\, q=q(\alpha,\beta,\gamma)\}$, where $q(\alpha,\beta,\gamma)$ is given by (\ref{pol_q}).
\item
\begin{multline*}
C=\{ 25, 115, 145, 199, 659, 731, 739, 925, 1195, 1235, 2381, 3259,
3365,\\
3709, 4705, 6325, 8989, 15095, 27991, 39239, 62129, 174641,
279199,\\281735, 310771, 404629, 1308259, 1822105, 2083075\}.
\end{multline*}
\end{itemize}
Of course this conjecture implies the $q-$Conjecture.

\section{A Greedy-Type Algorithm for the Erd\H{o}s-Straus Decomposition}\label{SectionAlgorithm}

Let us describe an algorithm to construct a decomposition of $\frac4n$ as sum of at most three Egyptian fractions. Using this algorithm we have checked ESC for $n$ up to $2\times10^{14}$. In Lemma \ref{lemConvAlgorPoli} bellow we prove that the following algorithm converges for all numbers $n=abc-a-b$ with $b\equiv c$ mod 4. Meantime in Lemma \ref{lemCondConvAlg} we give some conditions which characterize  the convergence of this algorithm.

\vspace*{0.5cm}

\noindent\textbf{A Greedy-Type Algorithm for the Erd\H{o}s-Straus Decomposition:} given $n\in\mathbb{N},$ $n\ge 2,$
\begin{enumerate}
\item[Step 1.] Set $j=1$ and $q=\lfloor \frac{n}{4}\rfloor$, the integer part of $\frac{n}{4}$.
\item[Step 2.] Set $x_j=q+j$, $\kappa_j=\frac4{n}-\frac1{x_j}$, and $y_j=\lceil \frac1{\kappa_j}\rceil$.
\item[Step 3.] If $\frac4{n}-\frac1{x_j}-\frac1{y_j}=0$ is true, the algorithm stops else we ask if $z_j=\frac1{\frac4{n}-\frac1{x_j}-\frac1{y_j}}\in\Z$ is true, then the algorithm stops also and $\frac4n=\frac1{x_j}+\frac1{y_j}+\frac1{z_j}$ else we set $j\hookleftarrow j+1$ and return to Step 2.
\end{enumerate}

Here $y=\lceil \frac1{\kappa_j}\rceil$ denotes the ceiling function evaluated at $\frac1{\kappa_j}$.

We call this algorithm ``The Greedy-type Algorithm for Erd\H{o}s-Straus decomposition.'' Observe that the algorithm finds the greedy decomposition of $\frac{4j-1}{nx_j},$
$j=1,2,\ldots$, and also stops when the decomposition is the sum of two Egyptian fractions. In fact,
if $r_j\equiv-nx_j$ (mod $4j-1$)  with $r_j\in [1,4j-2]$, then
\[
\frac{4j-1}{nx_j}=\frac1{y_j}+\frac{r_j}{nx_jy_j}.
\]
The algorithm stops when $r_j$ divides $nx_jy_j$ and $z_j=\frac{nx_jy_j}{r_j}$.

\begin{lem}\label{lemConvAlgorPoli} The greedy-type algorithm converges for all $n=abc-a-b$, with $bc\equiv 1$ mod $4$.
\end{lem}
\begin{proof} If $b\equiv c\equiv 1$ mod 4, then $n=abc-a-b\equiv 3$ mod 4, in this case $\frac4n$ has a decomposition as sum of two Egyptian fraction. So we only have to consider $b\equiv c\equiv 3$ mod 4.

First, we observe that if $(x,z)\in\Z_{>0}^2$, then $1/x$ is the Egyptian fraction nearest to $\frac1x+\frac1z$ and less than this number if and only if
\begin{equation}\label{EqJustGreedyAlg}
x(x-1)\le z.
\end{equation}
Moreover, we have
\begin{equation}\label{eqAuxJustGreedyAlg}
\frac4n=\frac4{abc-a-b}=\frac1{a\frac{bc-1}{4}}+\frac1{a(ac-1)\frac{bc-1}{4}}+\frac1{(ac-1)\frac{bc-1}{4}n},
\end{equation}
and
\[
a\frac{bc-1}{4}<a(ac-1)\frac{bc-1}{4}<(ac-1)\frac{bc-1}{4}p
\]
since $c=4z+3\ge3$, $ac-1\ge 2a$, and $n=abc-a-b\ge a\Leftrightarrow b(ac-1)\ge 2a$.

According to (\ref{EqJustGreedyAlg}) and (\ref{eqAuxJustGreedyAlg}), it is enough to prove
\begin{multline*}
a\frac{bc-1}{4}\left(a\frac{bc-1}{4}-1\right)\le (ac-1)\frac{bc-1}{4}n\\
\Leftrightarrow a\left(a\frac{bc-1}{4}-1\right)\le (ac-1)(abc-a-b).
\end{multline*}
Observe that both relations $a\left(a\frac{bc-1}{4}-1\right)=ba^2c/4-a^2/4-a$ and $(ac-1)^2b-a(ac-1)$ are linear function in $b$, so comparing their leading coefficients and their values at $b=1$, the above inequality follows.

\end{proof}

\begin{rmk} We are checked that the algorithm given above also works for Sierpi\'nski's Conjecture, $\frac5n=\frac1x+\frac1y+\frac1z$, taking $q=\lfloor \frac{n}{5}\rfloor$ and making other few trivial changes, for $n$ up to $10^{10}$.
\end{rmk}

If $n=4q+1$ and $(4q+1)(q+j)=(s+1)(4j-1)+r-(4j-1)$, then
\begin{equation}\label{ecuacAux7}
\frac4{4q+1}-\frac{1}{q+j}-\frac1{s+1}=\frac{4j-1-r}{(4q+1)(q+j)(s+1)}
\end{equation}
and
\begin{equation}\label{ecuacAux5}
(4q+1)^2(q+j)^2=(4q+1)(q+j)(s+1)(4j-1)+(r-(4j-1))(4q+1)(q+j).
\end{equation}
So the first part of following result follows:

\begin{lem}\label{lemCondConvAlg}   Set $n=4q+1\in\Z_{>0}$. Let $s=s(q,j)$ and $r=r(q,j)$ be such that $(4q+1)(q+j)=s(4j-1)+r$, $0\le r\le 4j-2$ for some $j\in\Z_{>0}$.
Then the greedy-type algorithm converges for $n=4q+1$ if and only if one of the following equivalent statement holds:
\begin{enumerate}
\item $(4j-1)-r$ divides to $(4q+1)(q+j)(s+1)$.
\item $(4j-1)-r$ divides to $(4q+1)^2(q+j)^2$.
\end{enumerate}
\end{lem}

\begin{proof}
According to (\ref{ecuacAux7}), it is enough to check that if $(4j-1)-r$ divides to $(4q+1)^2(q+j)^2$, then $(4j-1)-r$ also divides to $(4q+1)(q+j)(s+1)$. Let $\delta=(4j-1,s+1)$, so we have $(\alpha,\beta)=1$, where $\alpha=\frac{4j-1}{\delta}$ and $\beta=\frac{s+1}{\delta}$. Since $\delta$ is odd, $(4q+1)(q+j)=(4q+1)\frac{4q+1+(4j-1)}4$, and
\begin{equation}\label{ecuacAux6}
(4q+1)(q+j)(s+1)=(s+1)^2(4j-1)-(4k-1-r)(s+1)=(s+1)^2\delta\alpha-(s+1)\delta(\alpha-\beta),
\end{equation}
if $\delta$ divides $(4q+1)(q+j)(s+1)$, then $\delta^2$ divides $(4q+1)(q+j)(s+1)$. On the other hand, using (\ref{ecuacAux5}) and the hypothesis we obtain $\alpha-\beta$ divides $(4q+1)(q+j)(s+1)$. Therefore, from (\ref{ecuacAux6}) we conclude that $(4j-1)-r=\delta(\alpha-\beta)$ divides $(4q+1)(q+j)(s+1)$, and the lemma is proved.
\end{proof}

\begin{rmk}
For theoretical considerations is useful to observe that according to our notations $(4q+1)(q+j)=s(4j-1)+r\Leftrightarrow (4q+1)^2+(4q+1)(4j-1)=4s(4j-1)+4r$, so we have $(4q+1)^2=(4s-(4q+1))(4j-1)+4r\overset{\textnormal{def}}{=}(4k-1)(4j-1)+4r$ and $(4q+1)^2\equiv 4r$ mod $4j-1$.
\end{rmk}

\begin{rmk} From statement 2 in the lemma above it follows that if $n=4q+1$ is a prime, the greedy-type algorithm stops at a step $j$ with $4j-1<4q+1$ if and only if $4j-1-r$ divides $(q+j)^2$.
\end{rmk}

\begin{rmk} Given any $j\in\Z_{>0}$ there exists a prime number $n=n_j$ such that the number of steps in the algorithm above is larger than $j$. In fact, if 
\[
n\overset{\textnormal{def}}{=}4\, \textnormal{lcm}\{3,7\ldots,4j-1,2,5,\ldots,3j-1\}t+1\overset{\textnormal{def}}{=}4q_j+1
\]
is a prime for a given $t\in\Z_{>0}$, then for each $k\le j$, we have
\[
(4q_j+1)(q_j+k)\equiv k,\quad \textnormal{mod } 4k-1,
\]
and since $(k,3k-1)=1,$
\[
4k-1-k=3k-1\nmid (4q_j+1)^2(q_j+k)^2.
\]
Therefore, according to  the Lemma \ref{lemCondConvAlg}, the greedy-type algorithm does not converge in the first $j$ steps for $n_j.$

\end{rmk}

\section{Appendix. Computer Programs}

In this section we include some programs which are used to check the Erd\H{o}s-Straus Conjecture, $q-$Conjecture and the $q-$Strong Conjetures. For an independent reading we set again some facts already cited.

\subsection{A computer program checking the $q-$Strong Conjecture} Lemma \ref{lemNoSquares} and computer calculations let us to formulate the following conjecture which implies the ESC and has been checked up to  $2\times 10^{12}$:

{\bf $q-$Strong Conjecture:}  The set of positive integers, $\Z_{>0}$, is disjoint union of the following three sets:
\begin{itemize}
\item $A=\{q:\exists n\in\Z_{>0}, \, q=n^2+n-1\}$.
\item $B=\{q:\exists(\alpha,\beta,\gamma)\in\Z_{\ge0}^3,\, q=q(\alpha,\beta,\gamma)\}$, where
\begin{equation}\label{pol_q_2}
q=q(\alpha,\beta,\gamma)=((4\beta+3)\gamma+(3\beta+2))(\alpha+1)-(\beta+2),
\end{equation}
\item
\begin{multline*}
C=\{ 25, 115, 145, 199, 659, 731, 739, 925, 1195, 1235, 2381, 3259,
3365,\\
3709, 4705, 6325, 8989, 15095, 27991, 39239, 62129, 174641,
279199,\\281735, 310771, 404629, 1308259, 1822105, 2083075\}.
\end{multline*}
\end{itemize}
Of course this conjecture implies the $q-$Conjecture, see below and Lemma \ref{LemConjQ}. Next we describe our algorithm:
\begin{description}
\item[Step 1] We generate a matrix $w$ of equivalence classes in the set $B$ using the function {\it cribata} of two arguments. The first column of $w$ are the moduli and the second are their corresponding rests. We only use moduli which are divisors  of a given number. The values in $q$ of two of the three parameters $\alpha,$ $\beta,$ and $\gamma$, or their associated new parameters after a change of variable, are setting up to a bound. We use several equivalent expressions to the function (\ref{pol_q_2}):
    \begin{equation}\label{pol_q_3}
    q=4xyz+3xy+3xz+2x+4yz+2y+3z,
    \end{equation}
    \[
    q=x(4yz+3y+3z+2)+4yz+2y+3z,
    \]
    \[
    q=y(4xz+3x+4z+2)+3xz+2x+3z,
    \]
    \[
    q=z(4xy+3x+4y+3)+3xy+2x+2y,
    \]
    \[
    4q+6+x=(4y+3)(4(x+1)z+3x+2),
    \]
    \[
    q+y+2=(x+1)((4y+3)z+3y+2),
    \]
    \[
    (3+4z)q+4+5z=((3+4z)x+4z+2)((3+4z)y+3z+2),
    \]
    and others which are easy to find after a change of variables, setting $y=z+a$ in (\ref{pol_q_3}), we have
    \[
    q=x(4z^2+z(4a+6)+3a+2)+4z^2+z(4a+5)+2a,
    \]
    doing $z=\frac{u-a-1}{2}$ in (\ref{pol_q_3}) with $u$, $a$ of different parity,
$-u<a<u$,
\[
q=u^2+u-1-a^2-\frac{a+u+1}{2}+x(u^2+u-a^2),
\]
setting $a=u-2d-1$,
    \[
    q=(u-d-1)(4d+2)+3d+x(u(4d+3)-(2d+1)^2),\quad 0\le d\le \frac{u-1}{2},
    \]
    for $a=-u+2d+1$, we have
    \[
    q=(u-d-1)(4d+3)+2d+x(u(4d+3)-(2d+1)^2), \quad 0\le d\le \frac{u-1}{2},
    \]
    doing $x=s-t$, $y=t-z$ in (\ref{pol_q_3}), we obtain
    \[
    q=s(- 4 z^2  + 4 t z + 3 t + 2)+ 4 t z^2  - 4 z^2  - 4 t^2  z + 4 t z + z - 3
   t^2,
   \]
   doing the same $x=s-t$, $z=t-y$ in (\ref{pol_q_3}),
   \[
   q=(-4y^2+4ty+3t+2)s+4ty^2-4y^2  - 4 t^2  y + 4 t y - y - 3 t^2+t,
   \]
   with $y=s-t$, $z=t-x$ in (\ref{pol_q_3}),
   \[
   q=(- 4 x^2  + 4 t x - x + 4 t + 2)s+
   4 t x^2  - 3 x^2  - 4 t^2  x + 4 t x - x - 4 t^2  + t,
   \]
   doing $y=s-t$, $x=t-z$ in (\ref{pol_q_3}),
   \[
   q=s(- 4 z^2  + 4 t z + z + 3 t + 2)+
   4 t z^2  - 3 z^2  - 4 t^2  z + 2 t z + z - 3 t^2,
   \]
   setting $z=s-t$, $y=t-x$ in (\ref{pol_q_3}),
   \[
   q=s(- 4 x^2  + 4 t x - x + 4 t + 3+
   4 t x^2  - 3 x^2  - 4 t^2  x + 4 t x - 4 t^2  - t,
   \]
   with $z=s-t$, $x=t-y$ in (\ref{pol_q_3}),
   \[
   q=s(- 4 y^2  + 4 t y + y + 3 t + 3)+
   4 t y^2  - 3 y^2  - 4 t^2  y + 2 t y - 3 t^2  - t.
   \]
\item[Step 2] Next, we sieve $q$ in the residue classes  in the vector $w$ using the function \textit{cricu}.

The steps 1 and 2 are done two times. First, we use moduli which are  divisors  of the product of all primes up to 19. Next, we add  the factor 23 obtaining
\[
tc=2\times 3\times 5\times 7\times 11\times 13\times 17\times 19\times 23.
\]
\item[Step 3] As a by-product of steps above we have many rests modulo $tc$ which are not sieved, all these are saved in the vector $v$. The function {\it adjunta} attaches to a matrix a vector, this is also used in the Step 1 generating the matrix $w$. At the last, we use the function \textit{fasf} which combines square checking with functions \textit{cri} throughout  the equivalent classes in $v$ modulo $tc$.
\end{description}

Next, we show Sage code of our program, although we write also our algorithm in UBASIC and Pari GP. The Sage code is more transparent and easy to read. Moreover, the program have been run on the operating systems: MS-DOS, Linux, and OS X.

\begin{verbatim}

def adjunta(s,r,u):
    l=len(u)
    h=0
    for i in xsrange(l):
      if s%u[i][0]==0:
         c=gcd(s,u[i][0])
         if r%c==u[i][1]:
            h=1
    if h==0:
        l=l+1
        u.append([s,r])
    return(u)

def cribata(t,m):
    u=[[2,0],[3,0],[7,0]]
    m1=m+1
    for b in xsrange(m1):
        for a in xsrange(b+1):
            s=4*a*b+3*a+3*b+2
            if t%s==0:
                r=4*a*b+2*a+3*b
                u=adjunta(s,r,u)
            s=4*a*b+3*a+3*b+2
            if t%s==0:
                r=4*a*b+3*a+2*b
                u=adjunta(s,r,u)
            s=4*a*b+3*a+4*b+2
            if t%s==0:
                r=3*a*b+2*a+3*b
                u=adjunta(s,r,u)
            s=4*a*b+4*a+3*b+2
            if t%s==0:
                r=3*a*b+3*a+2*b
                u=adjunta(s,r,u)
            s=4*a*b+3*a+4*b+3
            if t%s==0:
                r=3*a*b+2*a+2*b
                u=adjunta(s,r,u)
            s=4*a*b+4*a+3*b+3
            if t%s==0:
                r=3*a*b+2*a+2*b
                u=adjunta(s,r,u)
            s=4*a*b-4*a^2+3*b+2
            if t%s==0:
                r=4*a^2*b-4*a^2-4*b^2*a+4*a*b-a-3*b^2+b
                u=adjunta(s,r,u)
            s=4*a*b-4*a^2+3*b+2
            if t%s==0:
                r=4*a^2*b-4*a^2-4*b^2*a+4*a*b+a-3*b^2
                u=adjunta(s,r,u)
            s=4*a*b-4*a^2+4*b+2-a
            if t%s==0:
                r=4*a^2*b-3*a^2-4*b^2*a+4*a*b-a-4*b^2+b
                u=adjunta(s,r,u)
            s=4*a*b-4*a^2+a+3*b+3
            if t%s==0:
                r=4*a^2*b-3*a^2-4*b^2*a+2*a*b-3*b^2-b
                if r>0:
                       u=adjunta(s,r,u)
            s=4*a*b-4*a^2+a+3*b+2
            if t%s==0:
                r=4*a^2*b-3*a^2-4*b^2*a+2*a*b-3*b^2+a
                if r>0:
                    u=adjunta(s,r,u)
            s=4*a*b-4*a^2-a+4*b+3
            if t%s==0:
                r=4*a^2*b-3*a^2-4*b^2*a+4*a*b-4*b^2-b
                if r>0:
                    u=adjunta(s,r,u)
            s=-4*a^2+4*a*b-4*a+3*b-1
            if s>0:
               if t%s==0:
                   r=- 4*a^2  + 4*a*b - 5*a + 3*b - 3
                   if r>0:
                       u=adjunta(s,r,u)
            s=-4*a^2+4*a*b-4*a+3*b-1
            if s>0:
               if t%s==0:
                   r= - 4*a^2  + 4*a*b - 3*a + 2*b - 2
                   if r>0:
                       u=adjunta(s,r,u)
    return(u)


t32=1;
t43=1;
for p in primes(2,500000):
    if p%3==2:
        t32=t32*p
    if p%4==3:
       t43=t43*p


def cri(q):
    j=-1
    h=0
    while h==0 and j<w1-1:
        j=j+1
        if q%(w[j][0])==w[j][1]:
            h=1
    return(h)

def fasf(q):
     a=0
     h=1
     if cri(q)==0:
        if is_square(4*q+5)==0:
           if gcd(2*q+3,t43)==1:
              if gcd(q+2,t32)==1:
                 if gcd(3*q+4,t32)==1:
                    q0=q+1
                    a1=-1
                    a0=-1
                    h=0
                    while h==0 and a0<(3*q/2):

                        q0=q0+1

                        a1=a1+4

                        a0=a0+3

                        for p in divisors(q0):

                            if p%a1==a0:

                                h=1

                    a=(a0-2)/3

     return([h,a])

def cricu(q):
    if gcd(q+2,t3)==1:
            if gcd(3*q+4,t3)==1:
                if gcd(2*q+3,t4)==1:
                    if gcd(q+3,t7)==1:
                        if gcd(7*q+9,t7)==1:
                            if (q+4)%19>0:
                                if (11*q+14)%19>0:
                                    if (q+5)%11>0:
                                        if (15*q+19)%11>0:
                                            if (q+7)%17>0:
                                                if (23*q+29)%17>0:
                                                    h=0
                                                    j=-1
                                                    while h==0 and j<len(w)-1:
                                                        j=j+1
                                                        if q%w[j][0]==w[j][1]:
                                                            h=1
                                                    if h==0:
                                                        v.append(q)
    return(v)


t=2*3*5*7*11*13*17*19
t3=2*5*11*17
t4=3*7*11*19
t7=5*19
w=cribata(t,2000)
w1=len(w)
print w1


v=[]
for q in xsrange(t):
    cricu(q)
v1=len(v)
print v1
save(v,"qf19")

t=2*3*5*7*11*13*17*19
tc=t*23
t3=2*5*11*17*23
t4=3*7*11*19*23
w=cribata(tc,2000)
w1=len(w)
print w1
v7=v
v=[]
v2=tc/t
a0=-1
for l1 in xsrange(0,v2):
    for i in xsrange(0,v1):
        q=v7[i]+l1*t
        cricu(q)
v1=len(v)
print v1
save(v,"qf23")

a0=0
for l1 in xsrange(0,9000):
    for i in xsrange(0,v1):
          q1=v[i]+l1*tc
          wt=fasf(floor(q1))
          a=wt[1]
          if wt[0]==0:
              print q1," falla "
              a=0
          if a>a0:
              a0=a
              print q1,a

\end{verbatim}

\subsection{A computer program verifying the $q-$Conjecture and Erd\H{o}s-Straus Conjecture of Type-I}

According to Lemma \ref{LemConjQ} if for each $q\in\Z_{>0}$ there exist $x,y,z$ in $\Z_{\ge0}$ such that one of the following relations holds
\begin{eqnarray*}
q=1+3x+3y+4xy,\\  q=5+5x+5y+4xy,\\ q=((4y+3)z+(3y+2))(x+1)-(y+2).
\end{eqnarray*}
the Erd\H{o}s-Straus Conjecture is true. We conjecture that all the natural number $q\in\Z_{>0}$ can be written as one of the three above relations. We refer it as the {\it $q-$Conjecture}.  Observe that this is equivalent to  for all prime numbers of the form $4q+5$ the corresponding value $q$ can be write as $q=((4y+3)z+(3y+2))(x+1)-(y+2)$ for some $x,y,z$ in $\Z_{\ge0}$. It is worst to observe that the identity
\[
\frac4{abc-a-b}=\frac{1}{a\frac{bc-1}4}+\frac{1}{a(ac-1)\frac{bc-1}4}
+\frac{1}{(ac-1)\frac{bc-1}4n},
\]
with $n=abc-a-b,$ implies that $\frac4n$ has a decomposition as sum of three Egyptian fraction with exactly one of the denominators multiplies of $n$ (a decomposition of type I).
We are verified $q-$Conjecture for $q\le 2.5\times 10^{13}$.
For this object, we do a litter change in the program above. We write two new functions {\it crice} and \textit{fase} instead of {\it cricu}  and {\it fasf}, respectively.

The Sage code of these function is:
\begin{verbatim}
def crice(q):
    if gcd(4*q+5,tc)==1:
        if gcd(q+2,t3)==1:
            if gcd(3*q+4,t3)==1:
                if gcd(2*q+3,t4)==1:
                    if gcd(q+3,t7)==1:
                        if gcd(7*q+9,t7)==1:
                            if (q+4)%19>0:
                                if (11*q+14)%19>0:
                                    if (q+5)%11>0:
                                        if (15*q+19)%11>0:
                                            if (q+7)%17>0:
                                                if (23*q+29)%17>0:
                                                    h=0
                                                    j=-1
                                                    while h==0 and j<len(w)-1:
                                                        j=j+1
                                                        if q%w[j][0]==w[j][1]:
                                                            h=1
                                                    if h==0:
                                                        v.append(q)
    return(v)


def fase(q):
     a=0
     h=1
     if cri(q)==0:
        if gcd(4*q+5,tf)==1:
           if gcd(2*q+3,t43)==1:
              if gcd(q+2,t32)==1:
                 if gcd(3*q+4,t32)==1:
                     if is_prime(4*q+5)==1:
                         q0=q+1
                         a1=-1
                         a0=-1
                         h=0
                         while h==0 and a0<(3*q/2):
                             q0=q0+1
                             a1=a1+4
                             a0=a0+3
                             for p in divisors(q0):
                                 if p%a1==a0:
                                     h=1
                         a=(a0-2)/3
     return([h,a])
\end{verbatim}

\subsection{A computer program checking the Erd\H{o}s-Straus Conjecture} We are verified that all the natural numbers  $n$, $2\le n\le 2\times 10^{14}$, satisfy the Ed\H{o}s-Straus Conjecture. To this end we do a more sharp sieve than in the above programs. The new {\it cricue} function replace {\it crice} function and add several new testing taking into account Lemma \ref{lemCondConvAlg} which gives a  characterization for the convergence of the greedy-type algorithm. Some of these new checkup are given by the minimal number of steps needed for the convergence of the greedy-type algorithm. Moreover, we divide the interval $[2.5\times 10^{13}, 5\times 10^{13}]$ in several intervals each one of $10^3$ numbers, then using the function {\it ferti}, which uses {\it fase} function, we check the ESC in each interval.

Next we show the Pari GP code of our program:

\begin{verbatim}
fase(q,h,q0,a1,a0,p,q7)={
     h=0;
     if(gcd(4*q+5,tp)>1,
         h=1,
     if(gcd(2*q+3,t43)>1,
         h=1,
         if(gcd(q+2,t32)>1,
             h=1,
             if(gcd(3*q+4,t32)>1,
                 h=1,
                 if(cri(q)==1,
                     h=1;
                     a0=2,
                     if(isprime(4*q+5)==0,
                         h=1;
                         a0=2,
                         q7=(3*q)\2;
                         q0=q+1;
                         a1=-1;
                         a0=-1;
                         h=0;
                         while(h==0&&a0<q7,
                             q0++;
                             a1+=4;
                             a0+=3;
                             fordiv(q0,p,
                                 if(p%a1==a0,
                                     h=1,)))))))));
      [h,(a0-2)\3]}

adjunta(s,r,u,c,h)={h=0;for(i=1,matsize(u)[1],if(s%u[i,1]==0,
                c=gcd(s,u[i,1]);if(r%c==u[i,2],
                 h=1,),));if(h==0,u=concat(u,[s,r]),);u}

{u=matrix(3,2,x,y,0);u[1,1]=2;u[2,1]=3;u[3,1]=7}
cribata(t,m,b,a,s,r)={
     for(b=0,m,
        for(a=0,b,
            s=4*a*b+3*a+3*b+2 ;
            if (t%s==0,
              r=4*a*b+2*a+3*b;u=adjunta(s,r,u),);
            s=4*a*b+3*a+3*b+2;
            if (t%s==0,
            r=4*a*b+3*a+2*b;u=adjunta(s,r,u),);
            s=4*a*b+3*a+4*b+2;
            if (t%s==0,
                r=3*a*b+2*a+3*b;
                u=adjunta(s,r,u),);
            s=4*a*b+4*a+3*b+2;
            if (t%s==0,
                r=3*a*b+3*a+2*b;
                u=adjunta(s,r,u),);
            s=4*a*b+3*a+4*b+3;
            if (t%s==0,
                r=3*a*b+2*a+2*b;
                u=adjunta(s,r,u),);
            s=4*a*b+4*a+3*b+3;
            if (t%s==0,
                r=3*a*b+2*a+2*b;
                u=adjunta(s,r,u),);
            s=4*a*b-4*a^2+3*b+2;
            if (t%s==0,
                r=4*a^2*b-4*a^2-4*b^2*a+4*a*b-a-3*b^2+b;
                if(r>0,
                       u=adjunta(s,r,u),),);
            s=4*a*b-4*a^2+3*b+2;
            if (t%s==0,
                r=4*a^2*b-4*a^2-4*b^2*a+4*a*b+a-3*b^2;
                if(r>0,
                       u=adjunta(s,r,u),),);
            s=4*a*b-4*a^2+4*b+2-a;
            if (t%s==0,
                r=4*a^2*b-3*a^2-4*b^2*a+4*a*b-a-4*b^2+b;
                if(r>0,
                       u=adjunta(s,r,u),),);
            s=4*a*b-4*a^2+a+3*b+3;
            if (t%s==0,
                r=4*a^2*b-3*a^2-4*b^2*a+2*a*b-3*b^2-b;
                if(r>0,
                       u=adjunta(s,r,u),),);
            s=4*a*b-4*a^2+a+3*b+2;
            if (t%s==0,
                r=4*a^2*b-3*a^2-4*b^2*a+2*a*b-3*b^2+a;
                if(r>0,
                       u=adjunta(s,r,u),),);
            s=4*a*b-4*a^2-a+4*b+3;
            if (t%s==0,
                r=4*a^2*b-3*a^2-4*b^2*a+4*a*b-4*b^2-b;
                if(r>0,
                       u=adjunta(s,r,u),),);
            s=-4*a^2+4*a*b-4*a+3*b-1;
            if(s>0,
               if (t%s==0,
                   r=- 4*a^2  + 4*a*b - 5*a + 3*b - 3;
                   if(r>0,
                       u=adjunta(s,r,u),),),);
            s=-4*a^2+4*a*b-4*a+3*b-1;
            if (s>0,
               if (t%s==0,
                   r= - 4*a^2  + 4*a*b - 3*a + 2*b - 2;
                   if(r>0,
                       u=adjunta(s,r,u),),),)));u}




cricue(q,h,j)={
    h=1;
    if(gcd(4*q+5,tc)==1,
    if(gcd(q+2,t3)==1,
    if(gcd(3*q+4,t3)==1,
    if(gcd(2*q+3,t4)==1,
    if(gcd(q+3,t7)==1,
    if(gcd(7*q+9,t7)==1,
    if((q+4)%19>0,
    if((11*q+14)%19>0,
    if((q+5)%11>0,
    if((15*q+19)%11>0,
    if((q+7)%17>0,
    if((23*q+29)%17>0,
    h=0;
    j=0;
    while(h==0 && j<matsize(w)[1],
       j++;
       if((q%w[j,1])==w[j,2],
       h=1,)),),),),),),),),),),),),);h}

{t32=1; t43=1;forprime(p=2,500000,if(p%3==2,
                                       t32*=p,);if(p%4==3,
                                                         t43*=p,))}

 cri(q,j,h)={
    j=0;
    h=0;
    while(h==0 && j<w1,
        j++;
        if(q%w[j,1]==w[j,2],
            h=1,));h}

{tc=2*3*5*7*11*13*17*19;
t=tc;
t3=2*5*11*17;
t4=3*7*11*19;
t7=5*19;
w=cribata(t,5000);
w1=matsize(w)[1];
print(w1)}
{v=[];
for(l=0,t\6-1,
    q=6*l+5;
    h=cricue(q);
    if(h==0,
        k=0;
        while(h==0&&k<5,
            k++;
            r=((4*q+5)*(q+1+k))%(4*k-1);
            c=((4*q+5)*(q+1+k))\(4*k-1);
            if(r==0,
                h=1,
                if(((4*q+5)*(q+1+k)*(c+1))%(4*k-1-r)==0,
                h=1,))),);
    if(h==0,v=concat(v,q),));
v1=matsize(v)[2];
print( " paso I ", v1)}
{v19=v;
    t=2*3*5*7*11*13*17*19;
    tc=t*23*9;
    t3=2*5*11*17*23;
    t4=3*7*11*19*23;
    w=cribata(tc,5000);
    w1=matsize(w)[1];
    print (w1);
    x=0;v23=vector(4000000,x);j23=0;
    v2=tc/t;}

{for(l1=0,v2-1,
        print l1;
        for(i19=1,v1,
        q=v19[i19]+l1*t;
        h=cricue(q);
        if(h==0,
            k=5;
            while(h==0 && k<7,
                k++;
                r=((4*q+5)*(q+1+k))%(4*k-1);
                c=((4*q+5)*(q+1+k)-r)\(4*k-1);
                if(r==0,
                    h=1,
                       if(((4*q+5)*(q+1+k)*(c+1))%(4*k-1-r)==0,
                           h=1,))),);
        if(h==0,
             j23++;v23[j23]=q,)));
        x=0;v=vector(j23,x);
 for(i23=1,j23,v[i23]=v23[i23]);
        v1=matsize(v)[2];
print("erti23 ",v1); write("v.txt",v)}

{tp=1;t32=1; t43=1;forprime(p=2,15000,tp*=p;
if(p%3==2,t32*=p,);if(p%4==3,t43*=p,))}

{t=2*27*5*7*11*13*17*19*23;
 w=cribata(t,100);
 w1=matsize(w)[1];
 print (w1)}

 ferti(li,lf,a0,l0,l,i7,q,y,a)={a0=0;l0=matsize(v)[2];
 for(l=li,lf,print(l);for(i7=1,l0,q=v[i7]+t*l;y=fase(q);a=y[2];
 if(y[1]==0,print(q," falla"); write("salida23",q,"falla"),
 if(a>a0,a0=a; print(" ",q," ",a);
 write ("salida23"," ",q," ",a),))))}
\end{verbatim}

\section{Acknowledgment}
We wish to thank J\'onathan Heras Vicente,   Jos\'e Antonio  Mart\'{\i}nez Mu\~noz and Juan Luis Varona, who allowed us to use their computers to perform calculations.

\end{document}